\documentclass[12pt]{amsart}
\usepackage{amssymb}

\usepackage[all]{xy}
\usepackage{amscd}

\nonstopmode

\newcommand{\HH}{H_{\mathfrak m}}
\newcommand{\fm}{\mathfrak m}
\newcommand{\Z}{\mathbb{Z}}
\newcommand{\R}{\mathbb{R}}
\newcommand{\N}{\mathbb{N}}

\newcommand{\ffi}{\varphi}

\newcommand{\be}{\beta}

\newcommand{\bl}{\boldsymbol{\ell}}

\DeclareMathOperator{\coker}{coker}

\DeclareMathOperator{\lk}{lk_{\Delta} }
\DeclareMathOperator{\st}{st_{\Delta} }

\DeclareMathOperator{\Soc}{Soc}

\DeclareMathOperator{\pnt}{\raise 0.5mm \hbox{\large\bf.}}

\newcommand{\s}{\; | \;}

\newcommand{\mif}{\mbox{if} ~}

\newcommand \pp[1] {^{\langle #1 \rangle}}

\newtheorem{theorem}{Theorem}[section]
\newtheorem{lemma}[theorem]{Lemma}
\newtheorem{proposition}[theorem]{Proposition}
\newtheorem{corollary}[theorem]{Corollary}

\newtheorem{lem-def}[theorem]{Lemma and Definition}
\newtheorem{prop-def}[theorem]{Proposition and Definition}

\theoremstyle{definition}
\newtheorem{remark}[theorem]{Remark}
\newtheorem{rem-def}[theorem]{Remark and Definition}

\newtheorem{notation}[theorem]{Notation}

\newtheorem*{acknowledgement}{Acknowledgement}
\title{Level algebras through Buchsbaum* manifolds}

\author[Uwe Nagel]{Uwe Nagel${}^{*}$}

\address{Department of Mathematics, University of Kentucky,
715 Patterson Office Tower, Lexington, KY 40506-0027, USA}
\email{uwenagel@ms.uky.edu}

\thanks{${}^*$ Part of the work for this paper was done while the
author was sponsored by the National Security Agency under Grant
Number H98230-09-1-0032.}

\begin{document}

\begin{abstract}
Stanley-Reisner rings of Buchsbaum* complexes are studied  by means
of their quotients modulo a linear system of parameters. The socle
of these quotients is computed. Extending a recent result by Novik and Swartz
for orientable homology manifolds without boundary, it is shown that
modulo a part of their socle these quotients are level algebras.
This provides new restrictions on the face vectors of Buchsbaum*
complexes.
\end{abstract}


\maketitle


%
%
%
\section{Introduction}

This note is inspired by work by Novik and Swartz \cite{NS-gor} and
by  Athanasiadis and Welker \cite{AW}. Its goal is to generalize
some results of \cite{NS-gor} and to contribute to the fruitful
interaction of algebraic, combinatorial, and topological methods in
order to study simplicial complexes.

Let $\Delta$ be a finite simplicial complex. Important algebraic
properties of it like Cohen-Macaulayness or Buchsbaumness are
defined by means of its Stanley-Reisner ring $K[\Delta]$. However,
these properties turn out to be topological properties in the sense
that  they depend only on the homeomorphism type of the geometric
realization $|\Delta|$ of $\Delta$. For example, all triangulations
of a manifold (with or without boundary) are Buchsbaum, and all
triangulations of a sphere or a ball are Cohen-Macaulay over each
field.

In \cite{AW}, Athanasiadis and Welker introduced Buchsbaum*
complexes  as the $(d-1)$-dimensional Buchsbaum complexes over a
field $K$ such that
\[
\dim_K \widetilde{H}_{d-2} (|\Delta| - p; K) = \dim_K \widetilde{H}_{d-2}(|\Delta|; K)
\]
for every $p \in |\Delta|$, where $\widetilde{H}_j (|\Delta|, K)$
denotes the reduced singular homology of the geometric realization
$|\Delta|$. The class of Buchsbaum* complexes includes all doubly
Cohen-Macaulay complexes and all triangulations of orientable
homology manifolds without boundary (see \cite{AW}).

In \cite{NS-Bu} Novik and Swartz pioneered the study of Buchsbaum
complexes by investigating socles of  artinian reductions of their
Stanley-Reisner rings. Our first main result is an improvement of
their description of the socle (\cite{NS-Bu}, Theorem 2.2) for
Buchsbaum* complexes. In fact, it provides a numerical
characterization of such complexes.

\begin{theorem}
  \label{thm:socle-Bu*}
Let $\Delta$ be a $(d-1)$-dimensional Buchsbaum simplicial complex.
Then $\Delta$ is a Buchsbaum* complex if and only if, for each
linear system of parameters $\bl$ of $K[\Delta]$ and each positive
integer $j$,
\[
\dim_K [\Soc K[\Delta]/\bl]_j = \binom{d}{j} \be_{j-1} (\Delta),
\]
where $\be_j (\Delta) := \dim_K \tilde{H}_j (\Delta; K)$ is the
$j$-th reduced Betti number of $\Delta$.
\end{theorem}

This extends \cite{NS-gor}, Theorem 1.3 because each triangulation
of an orientable  $K$-homology manifold without boundary is a
Buchsbaum* complex. Similarly, the following result extends
\cite{NS-gor}, Theorem 1.4.

\begin{theorem}
  \label{thm:intro-level}
Let $\Delta$ be a $(d-1)$-dimensional Buchsbaum* complex, and let
$\bl := \ell_1,\ldots,\ell_d$ be a linear system of parameters for
$K[\Delta]$. Set $I := \bigoplus_{j = 1}^{d-1} [\Soc
K[\Delta]/\bl]_j$. Then $(K[\Delta]/\bl)/I$ is a level ring of
Cohen-Macaulay type $\beta_{d-1} (\Delta)$ and socle degree $d$,
that is, the socle of $(K[\Delta]/\bl)/I$ is a $K$-vector space of
dimension $\beta_{d-1} (\Delta)$ that is concentrated in degree $d$.
\end{theorem}

Being a level ring provides strong restrictions on the Hilbert
function of $(K[\Delta]/\bl)/I$. However, the classification of
Hilbert functions of level algebras is a wide open problem, and the
above result lends further motivation to studying it.

This note is organized as follows. In Section \ref{sec-socles} we
study the socle of the artinian reduction of a Buchsbaum complex. We
observe that these complexes can in fact be characterized by their
socle (Corollary \ref{cor:Buchsb-comples-chara}). Moreover, we
slightly improve the description of this socle given in
\cite{NS-Bu}, Theorem 2.2, by identifying one more piece. Combined
with a result in \cite{AW}, this implies Theorem
\ref{thm:socle-Bu*}.

Section \ref{sec:level} is devoted to the proof of  Theorem
\ref{thm:intro-level}. It allows us to establish results on the
enumeration of the faces of Buchsbaum* complexes (Theorem
\ref{thm-upper-bound-h-Buchs*}) that improve the corresponding
results for arbitrary Buchsbaum complexes in \cite{NS-Bu}. This is
carried out in Section \ref{sec-enumeration}.

%
%

\section{Socles of artinian reductions}
\label{sec-socles}

Throughout this note $S := K[x_1,\ldots,x_n]$ denotes the polynomial
ring in $n$ variables over a field $K$.

Let $M = \oplus_{j\in \Z} [M]_j$ be a finitely generated, graded
$S$-module. Its $i$-th local cohomology module with support in the
maximal ideal $\fm :=(x_1,\ldots,x_n)$ is  denoted by $\HH^i (M)$
(see, e.g., \cite{BH-book} and \cite{ST96}. The {\em socle} of $M$
is the submodule
\[
\Soc M := 0 :_M \fm := \{y \in M \s \fm y = 0\}.
\]
The module $M (i)$ is the module with the same structure as $M$, but with a shifted grading defined by $[M (i)]_j := [M]_{i +j}$. Furthermore we use $s M$ to denote the direct sum of $s$ copies of $M$. If $M$ has (Krull) dimension $d$, a sequence $\ell_1,\ldots,\ell_d \in S$ of linear forms is called a {\em linear system of parameters} of $M$ if $M/\bl := M/(\ell_1,\ldots,\ell_d) M$ has dimension zero. In this case $M/\bl$ is called an {\em artinian reduction} of $M$.

Assume now that $M$ is a Buchsbaum module. For a comprehensive introduction to the theory of Buchsbaum modules we refer to \cite{St-V}. Here we only need the following facts about Buchsbaum modules:
\begin{itemize}
  \item For all $i \neq \dim M$, $\fm \HH^i (M) = 0$;

  \item If $\ell \in S$ is a linear parameter of $M$, that is $\dim M/ \ell M = \dim M -1$, then also $M/\ell M$ is Buchsbaum and the kernel of the multiplication map $M(-1) \stackrel{\ell}{\longrightarrow} M$   has Krull dimension zero.
\end{itemize}
Thus, the long exact cohomology sequence induced by the multiplication map splits into short exact sequences
\[
0 \to \HH^i (M) (-1) \to \HH^i (M/\ell M) \to H^{i+1} (M) \to 0 \quad \text{if }\; i+1 < \dim M =:d
\]
and ends with
\[
0 \to \HH^{d-1} (M) (-1) \to \HH^{d-1} (M/\ell M) \to \HH^d(M) (-1) \to \HH^d (M) \to 0.
\]
Using the first sequence repeatedly one obtains, for every part $\ell_1,\ldots,\ell_j$ of a linear system of parameters of $M$, the isomorphism of graded modules
\[
\HH^{i} (M/(\ell_1,\ldots,\ell_j)M) \cong \bigoplus_{k = 0}^{j} \binom{j}{k} \HH^{i + k} (M)(-k) \quad \text{if }\; i < d - j.
\]
In the case $i = 0$ and $j = d$, Novik and Swartz \cite{NS-Bu} established the following result on the socle of the module on the left-hand side.

\begin{theorem}[\cite{NS-Bu}, Theorem 2.2]
  \label{thm:socle-Bu}
Let $M$ be a finitely generated graded Buchsbaum $S$-module of
dimension $d$, and let $\bl := \ell_1,\ldots,\ell_d$ be a linear
system of parameters of $M$. Then
\[
\Soc M/\bl \cong \left ( \bigoplus_{j = 0}^{d-1} \binom{d}{j} \HH^j (M) (-j) \right )
\oplus Q (-d),
\]
where $Q$ is a graded submodule of $\Soc \HH^d (M)$.
\end{theorem}

In this section we make two comments about this result in the case
the module $M$ is a Stanley-Reisner ring.

A simplicial complex $\Delta$ on $n$ vertices is a collection of
subsets of $[n] := \{1,\ldots,n\}$ that is closed under inclusion.
Its Stanley-Reisner ideal is
\[
I_{\Delta} := (x_{j_1} \cdots x_{j_k} \s \{j_1 < \cdots < j_k\}
\notin \Delta) \subset S,
\]
and its Stanley-Reisner ring is $K[\Delta] := S/I_{\Delta}$. For
each subset $F \subset [n]$, the {\em link of $F$} is the subcomplex
\[
\lk F := \{G \in \Delta : F \cup G \in \Delta, F \cap G = \emptyset \}.
\]
Note that $\lk \emptyset = \Delta$.

The complex $\Delta$ is a Cohen-Macaulay or a Buchsbaum complex
over $K$ if $K[\Delta]$ has the corresponding  property.
Alternatively, a combinatorial-topological characterization of
Cohen-Macaulay complexes is due to Reisner \cite{Reisner}. Schenzel
extended his result in \cite{Sch}. A simplicial complex $\Delta$
 is Buchsbaum (over $K$) if and only if $\Delta$ is
pure and the link of each vertex of $\Delta$ is Cohen-Macaulay (over
$K$).

Our first observation about Theorem \ref{thm:socle-Bu} states  that
its description of the socle in fact characterizes Buchsbaum
simplicial complexes.

\begin{corollary}
  \label{cor:Buchsb-comples-chara}
Let $\Delta$ be a simplicial complex of dimension $d-1$, and let
$\bl := \ell_1,\ldots,\ell_d$ be a linear system of parameters of
its Stanley-Reisner ring $K[\Delta]$. Then $\Delta$ is Buchsbaum if
and only if
\[
\Soc K[\Delta]/\bl \cong \left ( \bigoplus_{j = 1}^{d-1} \binom{d}{j} \HH^j (K[\Delta]) (-j) \right )
\oplus Q (-d),
\]
where $Q$ is a graded submodule of $\Soc \HH^d (K[\Delta])$.
\end{corollary}

\begin{proof}
By Theorem \ref{thm:socle-Bu}, it is enough to show sufficiency. The
formula for the socle implies that, for each $j \in
\{1,\ldots,d-1\}$, the module $\HH^j (K[\Delta])$ is annihilated by
$\fm$ and finitely generated, hence it is a finite-dimensional
$K$-vector space. Therefore $\Delta$ must be a Buchbaum complex
(see, e.g., \cite{Sch} or \cite{NS-sing}).
\end{proof}

Our second observation identifies a piece of the module $Q$
occurring  in Theorem \ref{thm:socle-Bu}. In its proof we use
\[
e (P) := \sup \{j \in \Z \s [P]_j \neq 0\}
\]
to denote the end of a graded module $P$. Note that $e (N) =
-\infty$  if $N$ is trivial.

\begin{proposition}
  \label{prop:include-top-degree}
Let $\Delta$ be a Buchsbaum simplicial complex of dimension $d-1$, and let $\bl := \ell_1,\ldots,\ell_d$ be a linear system of parameters of  $K[\Delta]$. Then
\[
\Soc K[\Delta]/\bl \cong \left ( \bigoplus_{j = 1}^{d-1} \binom{d}{j} \HH^j (K[\Delta]) (-j) \right )
\oplus ([\HH^d (K[\Delta])]_0 S) (-d)  \oplus Q' (-d),
\]
where $[\HH^d (K[\Delta])]_0 S$ is the submodule of\ $\HH^d(K[\Delta])$ generated by its elements of degree zero and where $Q'$ is a graded submodule of\; $\Soc \HH^d (K[\Delta])$ that vanishes in all non-negative degrees.
\end{proposition}

\begin{proof}
Hochster's formula (see \cite{ST96}, Theorem 4.1)  provides that all
intermediate cohomology modules  $\HH^0 (K[\Delta]),\ldots,\HH^{d-1}
(K[\Delta])$ are concentrated in degree zero and that $\HH^d
(K[\Delta])$ vanishes in all positive degrees.

Set $\bl_j := \ell_1,\ldots,\ell_j$. Since  $\bl_d = \bl$ and
$\HH^{0} (K[\Delta]/\bl_d) \cong  K[\Delta]/\bl_d$, our claim
follows, once we have shown that,   for all $j \in \{0,\ldots,d\}$,
\[
e(\HH^{d-j} (K[\Delta]/\bl_j)) \le j \quad \text{and} \quad  \dim_K [\HH^{d-j} (K[\Delta]/\bl_j)]_j = \dim_K \HH^{d} (K[\Delta]).
\]

Indeed, if $j = 0$, this is true. Let $j < d$. Multiplication  by
$l_{j+1}$ on $K[\Delta]/\bl_j$ induces the long exact cohomology
sequence
\[
0 \to \HH^{d-j-1} (K[\Delta]/\bl_j) \to \HH^{d-j-1} (K[\Delta]/\bl_{j+1}) \to \HH^{d-j} (K[\Delta]/\bl_j) (-1) \to \HH^{d-j} (K[\Delta]/\bl_j).
\]
Since  $\HH^{d-j-1} (K[\Delta]/\bl_j) \cong \bigoplus_{i = 0}^{j}
\binom{j}{i} \HH^{d-j-1+i} (K[\Delta])(-i)$ vanishes in each degree
$k > j$, we get $e(\HH^{d-j-1} (K[\Delta]/\bl_{j+1})) \le j+1$ and
\[
\dim_K [\HH^{d-j-1} (K[\Delta]/\bl_{j+1})]_{j+1} =
\dim_K [\HH^{d-j} (K[\Delta]/\bl_j)]_j,
\]
as required.
\end{proof}

\begin{remark}
  \label{rem:contribution-zero-part}
The above argument shows that for an arbitrary, not necessarily Buchsbaum, simplicial complex $\Delta$,
the module $[\HH^d (K[\Delta])]_0 S$  contributes to the socle of $K[\Delta]/\bl$. This also follows from Lemma 2.3 in \cite{BN}.
\end{remark}

We now identify an instance where the module $Q'$ appearing in the previous result vanishes.

\begin{corollary}
  \label{cor:socle-reduction-b*}
Let $\Delta$ be a Buchsbaum* simplicial complex of dimension $d-1$, and let $\bl := \ell_1,\ldots,\ell_d$ be a linear system of parameters of  $K[\Delta]$. Then
\[
\Soc K[\Delta]/\bl \cong \left ( \bigoplus_{j = 1}^{d-1} \binom{d}{j} \HH^j (K[\Delta]) (-j) \right )
\oplus ([\HH^d (K[\Delta])]_0 S) (-d).
\]
\end{corollary}

\begin{proof}
According to \cite{AW}, Proposition 2.8, the socle of $\HH^d (K[\Delta]$ is concentrated in degree zero. Hence Proposition \ref{prop:include-top-degree} gives
the claim.
\end{proof}

The last result {\em proves Theorem \ref{thm:socle-Bu*}}  because
$\HH^j (K[\Delta]) = [\HH^j (K[\Delta])]_0$ if $ j \neq d$ and
$\dim_K [\HH^j (K[\Delta])]_0 = \be_{j-1} (\Delta)$ for all $j$ by
Hochster's formula.

%
%

\section{Level quotients}
\label{sec:level}

The goal of this section is to establish Theorem \ref{thm:intro-level}. Recall that an artinian graded $K$-algebra $A$ is a {\em level} ring if its socle is concentrated in one degree, that is, $[\Soc A]_j = 0$ if $j \neq e(A)$.

Let $\Delta$ be a simplicial complex on $[n]$. Recall  that each
subset $F \subseteq [n]$ induces the  following simplicial
subcomplexes of $\Delta$: the {\em star}
\[
\st F := \{G \in \Delta : F \cup G \in \Delta\},
\]
and the {\em deletion}
$$
\Delta_{-F} := \{G \in \Delta : F \cap G = \emptyset  \}.
$$
If $F \notin \Delta$, then $\lk F = \st F = \emptyset$.

Consider any vertex $k \in [n]$. Then $\st k$ is the cone over the
link $\lk k$ with apex $k$. Hence its Stanley-Reisner ideal is
$I_{\st k} = I_{\Delta} : x_k$. Furthermore,  the Stanley-Reisner
ideal of the deletion $\Delta_{-k}$ considered as a complex on $[n]$
is $(x_k, J_{\Delta_{-k}}) = (x_k, I_{\Delta})$, where
$J_{\Delta_{-k}} \subset S$ is the extension ideal of the
Stanley-Reisner ideal of $\Delta_{-k}$ considered as a  complex on
$[n] \setminus \{k\}$. Thus, we get the short exact sequence.
\begin{equation} \label{eq:star-del-seq}
0 \to (K[\st k]) (- \deg x_k) \overset{x_k}{\longrightarrow} K[\Delta] \to K[\Delta_{-k}] \to 0.
\end{equation}

As preparation, we need the following technical result.

\begin{lemma}
  \label{lem:link-del-seq}
Let $\Delta$ be a $(d-1)$-dimensional Buchsbaum complex, and let $k$ be a vertex of $\Delta$ such that $\lk k$ is 2-CM and $\Delta{_k}$ is Buchsbaum of dimension $d-1$. Then,  for every linear system of parameters $\bl := \ell_1,\ldots,\ell_d$ of
$K[\Delta]$, there is an exact sequence of graded modules
\[
0 \to (K[\st k]/\bl) (-1) \to K[\Delta]/\bl \to K[\Delta_{-k}]/\bl \to 0,
\]
where the first map is induced by multiplication by $x_k$.
\end{lemma}

\begin{proof}
For $j \in [d]$, set $\bl_j := \ell_1,\ldots,\ell_j$, so $\bl_d =
\bl$.  We will show that there is a short exact sequence
\begin{equation}
  \label{eq:induction-claim}
0 \to (K[\st k]/\bl_j) (-1) \to K[\Delta]/\bl_j \to K[\Delta_{-k}]/\bl_j \to 0.
\end{equation}

Since $\Delta$ is Buchsbaum, the long exact cohomology sequence
induced by Sequence \eqref{eq:star-del-seq} provides the exact
sequence
\begin{equation*}
  0 \to \HH^{d-1}(K[\Delta]) \stackrel{\ffi}{\longrightarrow} \HH^{d-1} (K[\Delta_{-k}]) \to \HH^d (K[\st k])(-1)
\end{equation*}
Using that the modules $\HH^{d-1}(K[\Delta])$ and
$\HH^{d-1}(K[\Delta_{-k}])$ are  concentrated in degree zero, we see that every
non-trivial element in $\coker \ffi$ gives a socle element of
$\HH^{d} (K[\st k])$. However, since $\lk k$ is 2-CM, the socle of
$\HH^{d} (K[\st k])(-1)$ is concentrated in degree $d >0$. We
conclude that $\ffi$ is an isomorphism.

Since $\st k$ is Cohen-Macaulay and has dimension $d-1$, it follows
that
\[
\HH^i (K[\Delta]) \cong \HH^i (K[\Delta_{-k}]) \quad
\text{for all} \; i \neq d.
\]

Using that $\Delta$ and $\Delta_{-k}$ are Buchsbaum this implies,  for
$j = 0,\ldots,d-1$,
\begin{equation}
  \label{eq:compare-local-coho}
\HH^0 (K[\Delta]/\bl_j) \cong \HH^0 (K[\Delta_{-k}]/\bl_j).
\end{equation}

We now show exactness  of Sequence \eqref{eq:induction-claim} by
induction.  Let $j \in \{0,\ldots,d-1\}$. Using the induction
hypothesis, multiplication by $\ell_{j+1}$ induces the following
commutative diagram
\begin{equation*}
\begin{CD}
&&&& 0 && 0 \\
&&&& @VVV @VVV  \\
&& 0 &&  0 :_{K[\Delta]/\bl_j} \ell_{j+1} &&  0 :_{K[\Delta_{-k}]/\bl_j} \ell_{j+1} \\
&& @VVV @VVV @VVV \\
0 @>>> (K[\st k]/\bl_j)(-1)  @>>> K[\Delta]/\bl_j  @>>> K[\Delta_{-k}]/\bl_j  @>>> 0 \\
&& @VV{\ell_{j+1}}V  @VV{\ell_{j+1}}V @VV{\ell_{j+1}}V  \\
0 @>>> K[\st k]/\bl_j  @>>> (K[\Delta]/\bl_j) (1) @>>> (K[\Delta_{-k}]/\bl_j)(1) @>>> 0 \\[1ex]
\end{CD}
\end{equation*}
Since $K[\Delta]/\bl_j$ and $K[\Delta_{-k}]/\bl_j$ are Buchsbaum rings, we conclude that
\[
0 :_{K[\Delta]/\bl_j} \ell_{j+1} \cong \HH^0 (K[\Delta]/\bl_j) \cong  \HH^0 (K[\Delta_{-k}]/\bl_j)
\cong 0 :_{K[\Delta_{-k}]/\bl_j} \ell_{j+1}.
\]
Hence, the Snake lemma provides the exact sequence
\[
0 \to (K[\st k]/\bl_{j+1}) (-1) \to K[\Delta]/\bl_{j+1} \to K[\Delta_{-k}]/\bl_{j+1} \to 0,
\]
as desired.
\end{proof}

Note that the above result remains true if one uses a  system of
parameters of arbitrary (positive) degrees. In the special case when $\Delta$ is an orientable homology manifold, Lemma \ref{lem:link-del-seq} essentially reduces to \cite{Swartz}, Proposition 4.24.

The following result is more general than Theorem
\ref{thm:intro-level}. However, to take full advantage of this
generality one needs an analogue  of Theorem \ref{thm:socle-Bu*} for
2-Buchsbaum complexes.

\begin{theorem}
  \label{thm:level}
Let $\Delta$ be a $(d-1)$-dimensional 2-Buchsbaum simplicial complex such that $\beta_{d-1} (\Delta) \neq 0$, and let
$\bl := \ell_1,\ldots,\ell_d$ be a linear system of parameters of
$K[\Delta]$. Set $I := \bigoplus_{j = 1}^{d-1} [\Soc
K[\Delta]/\bl]_j$. Then $(K[\Delta]/\bl)/I$ is a level ring of
Cohen-Macaulay type $\beta_{d-1} (\Delta)$.
\end{theorem}

\begin{proof}
Since we know that $\dim_K [K[\Delta]/\bl]_d = \beta_{d-1} (\Delta)$
is it enough to show that the socle of $(K[\Delta]/\bl)/I$ vanishes
in all degree $j < d$. By definition of $I$, this is true if $j =
d-1$. Let $z \in K[\Delta]/\bl$ be an element of degree $j \le d-2$
such that $\fm \cdot z \subseteq I \subseteq \Soc K[\Delta]/\bl$. We
have to show that $z$ is already in $\Soc K[\Delta]/\bl$.

To this end let $k \in [n]$ be any vertex of $\Delta$.  Since
$\Delta$ is 2-Buchsbaum, $\lk k$ is 2-CM by Miyazaki (\cite{Mi},
Lemma 4.2). Thus, we may apply Lemma \ref{lem:link-del-seq}. We
rewrite the exact sequence therein as
\[
0 \to (K[\st k]/\bl) (-1) \stackrel{\ffi}{\longrightarrow}
K[\Delta]/\bl \to K[\Delta]/(\bl, x_k) \to 0.
\]
This sequence implies that there is some $y \in K[\st k]/\bl$ of
degree $j$ such that $\ffi (y) = x_k \cdot z$. Since $x_k \cdot z$
is in $\Soc K[\Delta]/\bl$, we conclude that $y$ is in the socle of
$K[\st k]/\bl)$. However, since $\lk k$ is 2-CM of dimension $d-2$
and $\st k$ is the cone over $\lk k$, it follows that the socle of
$K[\st k]/\bl$ is concentrated in degree $d-1 > j = \deg y$. This
shows that $y = 0$, thus $x_k \cdot z = 0$. Since this is true for
every vertex $k$, we have shown $\fm \cdot z = 0$, as required.
\end{proof}

%
%

\section{Face enumeration}
\label{sec-enumeration}

Now we discuss how Theorems \ref{thm:socle-Bu*} and
\ref{thm:intro-level} imply upper and lower bounds on the face
vector of a Buchsbaum* complex $\Delta$.

The face or $f$-vector of a $(d-1)$-dimensional simplicial complex $\Delta$
is  the sequence $f (\Delta) := (f_{-1} (\Delta),  f_{0}
(\Delta),\ldots,f_{d-1} (\Delta))$, where $f_{j} (\Delta)$ is the
number of $j$-dimensional faces of $\Delta$. The same information is
encoded in the $h$-vector $h (\Delta) := (h_{0}
(\Delta),\ldots,h_{d} (\Delta))$, which is defined by
\[
\frac{h_0 (\Delta) + h_1 (\Delta) t + \cdots + h_d (\Delta) t^d}{(1-t)^d}
:=
\sum_{j \ge 0} \dim_K [K[\Delta]]_j t^j =
\sum_{j= -1}^{d-1} \frac{f_j (\Delta) t^{j+1}}{(1-t)^{j+1}}.
\]
More explicitly, this gives
\[
h_j (\Delta) = \sum_{i=0}^j (-1)^{j-i} \binom{d-i}{j-i} f_{i-1} (\Delta) \quad \text{and} \quad
f_{j-1} (\Delta) = \sum_{i=0}^j \binom{d-i}{j-i} h_i (\Delta).
\]
The  $h'$-vector $h' (\Delta) := (h'_{0} (\Delta),\ldots,h'_{d}
(\Delta))$  is defined by
\[
h'_j (\Delta) := h_j (\Delta) + \binom{d}{j} \sum_{i=0}^{j-1} (-1)^{j-i-1} \be_{i-1} (\Delta).
\]
If $\Delta$ is Buchsbaum, then its $h'$-vector is again a Hilbert
function  because, according to  \cite{Sch},
\[
h'_j = \dim_K [K[\Delta]/\bl]_j
\]
if $\bl := \ell_1,\ldots,\ell_d$ is a linear system of parameters of
$K[\Delta]$.

Following \cite{NS-Bu}, we define the $h''$-vector $h'' (\Delta)  :=
(h''_{0} (\Delta),\ldots,h''_{d} (\Delta))$ of $\Delta$ by
\[
h''_j (\Delta) := h'_j (\Delta) - \binom{d}{j} \be_{j-1} (\Delta) =
h_j (\Delta) + \binom{d}{j} \sum_{i=0}^{j} (-1)^{j-i-1} \be_{i-1} (\Delta)
\quad \text{if} \; i < d
\]
and
\[
h''_d (\Delta) := \be_{d-1} (\Delta).
\]
The key for our purposes is that in case
 $\Delta$ is Buchsbaum* $h''(\Delta)$ also is a Hilbert function.

\begin{corollary}
  \label{cor:Hilb-funct-level}
Let $\Delta$ be a $(d-1)$-dimensional Buchsbaum* complex. Using the
notation of Theorem \ref{thm:intro-level} set $\overline{K(\Delta)}
:= (K[\Delta]/\bl)/I$. Then $\overline{K(\Delta)}$ is a level
algebra whose Hilbert function is given by
\[
\dim_K [\overline{K(\Delta)}]_j = h''_j (\Delta).
\]
\end{corollary}

\begin{proof}
This follows by combining Theorem \ref{thm:socle-Bu*} and Theorem
\ref{thm:intro-level}.
\end{proof}

In order to use this information, we recall a result about Hilbert
functions.

\begin{notation} \label{not-bin-expansion}
(i) We always use the following convention for binomial
coefficients: If $a \in \R$ and $j \in \Z$ then
$$
\binom{a}{j} := \left \{ \begin{array}{ll}
\frac{a (a-1) \cdots (a-j+1)}{j!} & \mif j > 0 \\
1 & \mif j = 0 \\
0 & \mif j < 0.
\end{array} \right.
$$

(ii) Let $b > 0$ and $d \ge 0$ be  integers. Then there are uniquely
determined integers $m_s,\ldots,m_{d+1}$ such that $m_{d+1} \ge 0$,
\  $n+d -2 \ge m_d
> m_{d-1} > \ldots
> m_s \geq s \geq 1$, and
$$
b = m_{d+1} \binom{n-1+d}{d} + \binom{m_d}{d} + \binom{m_{d-1}}{d-1} +
\ldots + \binom{m_s}{s}.
$$
This is called the {\em $d$-binomial expansion} of $b$. For any
integer $j$ we set
\[
b\pp{d} :=  m_{d+1} \binom{n+d}{d+1} + \binom{m_d + 1}{d + 1} +
\binom{m_{d-1} + 1}{d}
+ \ldots +
\binom{m_s + 1}{s + 1}.
\]

(iii) If $b = 0$, then we put $b\pp{d}  := 0$ for all $d \in \Z$.
\end{notation}
\medskip

Assume $b >0$. Note that then $s = d+1$ if and only if
$\binom{n-1+d}{d}$ divides $b$. Furthermore, $\binom{m_d}{d} +
\binom{m_{d-1}}{d-1} + \ldots + \binom{m_s}{s}$ is the standard
$d$-binomial representation of $b - m_{d+1} \binom{n-1+d}{d}$.
\smallskip

We are ready to state a generalization of Macaulay's
characterization of Hilbert functions of algebras to modules, which
has been proven by Hulett \cite{Hulett} in characteristic zero and
by Blancafort and  Elias  \cite{BE} in arbitrary characteristic.

\begin{theorem}
  \label{thm-chara-Hilb-funct-modules}
For a numerical function $h: \N_0 \to \N_0$, the following
conditions are equivalent:
\begin{itemize}
  \item[(a)] In non-negative degrees $h$ is the Hilbert function of
  a module over $S = K[x_1,\ldots,x_n]$ that is generated in
  degree zero, i.e., there is a graded  $S$-module $M$ whose
  minimal generators have degree zero such that $h(j) = \dim_K
  [M]_j$ whenever $j \ge 0$.

  \item[(b)] For all integers $j \ge 0$,
\[
h (j+1) \le h(j)\pp{j}.
\]
\end{itemize}
\end{theorem}

\begin{proof}
  This follows from  \cite{BE}, Theorem 3.2.
\end{proof}

The following result provides restrictions on the face vectors of
Buchsbaum* complexes with given Betti numbers. The upper bound strengthens
\cite{NS-Bu}, Theorem 4.3, for Buchsbaum complexes in the case of
Buchsbaum* complexes.

\begin{theorem}
  \label{thm-upper-bound-h-Buchs*}
Let $\Delta$ be a $(d-1)$-dimensional Buchsbaum* complex on $n$
vertices. Then its $h'$-vector $(h'_0,\ldots,h'_d)$ and $h''$-vector $(h''_0,\ldots,h''_d)$ satisfy
$h'= h''_0  = 1$, $h'_1 = h''_1  = n-d$,  $h''_{d}  = \be_{d-1} (\Delta)$, and
\begin{itemize}
  \item[(a)] \[
h'_{j+1}  \le \min \left \{(h''_{j})\pp{j}, (h''_{j+2})\pp{d-j-2} + \be_j (\Delta) \binom{d}{j+1} \right \}
\quad \text{if } \; 1 \le j \le d-2;
\]
  \item[(b)] \[
  h''_{d-j} \ge \frac{h''_j}{\be_{d-1} (\Delta)} \quad \text{if
} \; 1 \le j \le d-1.
\]
\end{itemize}

\end{theorem}

\begin{proof}
(a) We use the notation of Corollary \ref{cor:Hilb-funct-level}.
The inequality $h'_{j+1}  \le (h''_{j})\pp{j}$ follows as in \cite{NS-Bu}, Theorem 4.3, by applying Theorem \ref{thm-chara-Hilb-funct-modules} to the algebra $(K[\Delta]/\bl)/[I]_j S$.

The canonical module
$\omega_{\overline{K(\Delta)}}$ of $\overline{K(\Delta)} = (K[\Delta]/\bl)/I$ is generated in degree $-d$
because $\overline{K(\Delta)}$ is level. Notice that, for all
integers $j$,
\[
\dim_K [\omega_{\overline{K(\Delta)}}]_{-j} =
\dim_K [\overline{K(\Delta)}]_j =
h''_j.
\]
Hence Theorem \ref{thm-chara-Hilb-funct-modules} provides
\[
h'_{j+1} - \be_j (\Delta) \binom{d}{j+1} =  h''_{j+1}  \le (h''_{j+2})\pp{d-j-2},
\]
which completes the proof of (a).

(b) is a consequence of Theorem 2 in \cite{St-CM-complexes}.
\end{proof}

\begin{remark}
  (i) It is a wide open problem to characterize the
  Hilbert functions of artinian level algebras. A systematic
  study of them was begun in \cite{GHMS}.

  (ii) The bound in Part (b) is never attained. This will be shown in
  the forthcoming paper \cite{BMMNZ}.

  (iii) According to S\"oderberg  (\cite{Soed}, Theorem 4.7), the
  $h''$-vector also has to satisfy the determinantal condition
  \[
     \left | \begin{array}{ccc}
     h''_{j-1} & h''_{j}  & h''_{j+1} \\
     r_{j-1} & r_{j} & r_{j+1} \\
     r_{d-j+1} & r_{d-j} & r_{d-j-1}
     \end{array}
     \right | \ge 0
  \]
  for all integers $j$, where $r_j := \binom{n-1+j}{j}$.
  S\"oderberg's result depends on a joint conjecture with Boij in \cite{BS} that
  has been proven by Eisenbud and Schreyer  \cite{ES}.

  Notice that S\"oderberg's condition characterizes the $h$-vectors
  of artinian level algebras up to multiplication by a rational
  number.
\end{remark}

\begin{acknowledgement}
    The author would like to thank Isabella Novik, Christos
    Athana\-siadis,
   and an anonymous referee  for helpful comments.
\end{acknowledgement}

%
%


\begin{thebibliography}{99}

\bibitem{AW}
C.\ A.\ Athanasiadis, V.\ Welker, {\em Buchsbaum* complexes},
Preprint, 2009; also available at arXiv:0909.1931.

\bibitem{BN}
E.\ Babson, I.\ Novik, {\em Face numbers and nongeneric initial ideals}, Electron.\ J.\ Comb.\ {\bf 11} (2006), Special volume in honor of R.\ Stanley,  \#R25, 23 pp.


\bibitem{BE}
C.\ Blancafort, J.\ Elias, {\em On the growth of the Hilbert
function of a module},  Math.\ Z.\ {\bf 234}  (2000), 507--517.


\bibitem{BS}
M.\ Boij, J.\ S\"oderberg, {\em Graded Betti numbers of Cohen-Macaulay modules and
the multiplicity conjecture}, J.\ London Math.\ Soc.\ {\bf 78}
 (2008), 78--101.

\bibitem{BMMNZ} M.\ Boij, J.\ Migliore, R.\ Mir\'o-Roig, U.\ Nagel,
F.\ Zanello, in preparation.


\bibitem{BH-book}
W.\ Bruns, J.\ Herzog,
{\em Cohen-Macaulay rings. Rev. ed.}.
Cambridge Studies in Advanced Mathematics {\bf 39},
Cambridge University Press, 1998.


\bibitem{ES}
D.\ Eisenbud, F.\ Schreyer, {\em Betti numbers of graded modules and
cohomology of vector bundles},  J.\ Amer.\ Math.\ Soc.\ {\bf 22}
(2009),  859--888.

\bibitem{GHMS}
A.V.\ Geramita, T.\ Harima, J.\ Migliore, and Y.S.\ Shin, {\em The
Hilbert function of a level algebra}, Mem.\ Amer.\ Math.\ Soc.\ {\bf
186} (2007),  No.\ 872.\

\bibitem{Hulett}
H.\ Hulett, {\em A generalization of Macaulay's theorem}, Comm.\
Algebra {\bf 23}  (1995), 1249--1263.

\bibitem{NS-sing}
E.\ Miller, I.\ Novik, E.\ Swartz, {\em Face rings of simplicial
complexes with singularities}, Preprint, 2010; also available at
arXiv:1001.2812.

\bibitem{Mi}
M.\ Miyazaki, {\em On 2-Buchsbaum complexes},  J.\ Math.\ Kyoto
Univ.\ {\bf 30} (1990),  367--392.


\bibitem{NS-Bu}
I.\ Novik, E.\ Swartz, {\em Socles of Buchsbaum modules, complexes
and posets}, Adv.\ M.\ {\bf  222}  (2009),  2059--2084.

\bibitem{NS-gor}
I.\ Novik, E.\ Swartz, {\em Gorenstein rings through face rings of
manifolds},  Compos.\ Math.\ {\bf 145}  (2009),  993--1000.


\bibitem{Reisner}
G.\ Reisner, {\em Cohen-Macaulay quotients of polynomial rings},
Adv.\ Math.\ {\bf 21} (1976), 30--49.

\bibitem{Sch}
P.\ Schenzel, {\em On the number of faces of simplicial complexes
and  the purity of Frobenius}, Math.\ Z.\ {\bf 178} (1981),
125--142.

\bibitem{Soed}
J.\ S\"oderberg,  {\em Graded Betti numbers and $h$-vectors of level
modules}, Preprint, 2006; also available at arXiv:0803.1645.

\bibitem{ST96}
R.\ P.\ Stanley, {\em Combinatorics and commutative algebra. Second
edition}. Progress in Mathematics {\bf 41}, Birkh\"auser Boston
1996.

\bibitem{St-CM-complexes} R.\ P.\  Stanley, {\em Cohen-Macaulay Complexes},
Higher Combinatorics, M.\ Aigner Ed., Reidel, Dordrecht and Boston
(1977), 51--62.

\bibitem{St-V}
J. St\"uckrad, W. Vogel: {\emph{Buchsbaum rings and applications.
An interaction between algebra, geometry and topology}},
Springer-Verlag, Berlin, 1986.

\bibitem{Swartz}
E.\ Swartz, {\em Face enumeration: From spheres to manifolds}, J.\ Eur.\ Math.\
Soc.\ {\bf 11} (2009), 449--485.



\end{thebibliography}
\end{document}